\documentclass[11pt]{article}

\usepackage[T1]{fontenc}
\usepackage{lmodern}
\usepackage{amsmath,amssymb,amsthm,amsfonts,mathrsfs}
\usepackage{geometry}
\usepackage{bm}
\usepackage{color}
\usepackage{microtype}
\usepackage[hidelinks]{hyperref}
\numberwithin{equation}{section}

\newtheorem{theorem}{Theorem}[section]
\newtheorem{lemma}{Lemma}[section]

\theoremstyle{remark}
\newtheorem{remark}{Remark}[section]

\newcommand{\doi}[1]{\href{https://doi.org/#1}{doi:\,\nolinkurl{#1}}}
\newcommand{\arxiv}[1]{\href{https://arxiv.org/abs/#1}{arXiv:\nolinkurl{#1}}}

\title{A Spectral Approach to Survival Bounds for the Elephant Random Walk }

\author{
Cristian Favio Coletti\thanks{
cristian.coletti@ufabc.edu.br\\
UFABC -- Centro de Matem\'atica, Computa\c{c}\~ao e Cogni\c{c}\~ao\\
Avenida dos Estados, 5001, 
Bangu, Santo Andr\'e -- SP, Brazil
}
\and
Rafael de Mattos Grisi\thanks{
rafael.grisi@ufabc.edu.br\\
UFABC -- Centro de Matem\'atica, Computa\c{c}\~ao e Cogni\c{c}\~ao\\
Avenida dos Estados, 5001, 
Bangu, Santo Andr\'e -- SP, Brazil
}
\and
Ioannis Papageorgiou\thanks{
i.papageorgiou@ufabc.edu.br, 
papyannis@yahoo.com\\
UFABC -- Centro de Matem\'atica, Computa\c{c}\~ao e Cogni\c{c}\~ao\\
Avenida dos Estados, 5001, 
Bangu, Santo Andr\'e -- SP, Brazil
}
}

\date{}

\begin{document}

\maketitle

\begin{abstract}
We study a rate-one Poissonization of the one-dimensional elephant random walk killed upon reaching $\pm N$. For every fixed memory parameter $0<p<1$, we give a self-contained proof of two-sided exponential survival bounds on the rate scale $N^{-2}$, uniformly for sufficiently large $N$ and times $t\ge TN^3$. Our approach combines the classical Dirichlet Poincar\'e inequality with an $O(N/t)$ estimate on the memory-dependent perturbation and a quantitative initialization of the principal-eigenvector projection. The lower bound covers the critical and superdiffusive regimes $3/4\le p<1$ in this late-time range. The confined-path comparison used for initialization also gives a direct proof of the lower bound. The spectral argument controls the exponentially small forcing terms and jumps introduced by Poisson truncation.
\end{abstract}



\section{Introduction and main result}

The elephant random walk (ERW) is a nearest-neighbor random walk whose increments depend on its complete past. Let $Y_0=0$ and
\[
Y_n=\sum_{j=1}^n\sigma_j,\qquad \sigma_j\in\{-1,1\}.
\]
Fix $p\in(0,1)$, put $\rho=2p-1\in(-1,1)$, and prescribe
$\mathbb P(\sigma_1=1)=r\in[0,1]$. For $n\ge2$, choose a past index uniformly from $\{1,\ldots,n-1\}$, repeat its increment with probability $p$, and reverse it otherwise. With $\mathcal F_n=\sigma(\sigma_1,\ldots,\sigma_n)$, the increment rule gives, for $n\ge1$,
\begin{equation}\label{eq:skeleton}
\mathbb P(\sigma_{n+1}=\pm1\mid\mathcal F_n)
=\frac12\pm\frac{\rho Y_n}{2n}.
\end{equation}
Thus the discrete position process is a time-inhomogeneous Markov chain, despite its memory-based construction.

For an integer $N\ge2$, let
\[
\nu_N=\inf\{n\ge0:|Y_n|=N\}.
\]
Let $(n(t))_{t\ge0}$ be an independent rate-one Poisson process, and define the absorbed process and its absorption time by
\[
X_t=Y_{n(t)\wedge\nu_N},\qquad
\tau_N=\inf\{t\ge0:|X_t|=N\}.
\]
Its transient state space is $D_N=\{-N+1,\ldots,N-1\}$. The pair consisting of position and Poisson count is Markov; the Poissonized position alone need not be Markov. Our forward operators below describe its marginal probabilities.

The ERW was introduced by Sch\"utz and Trimper \cite{SchutzTrimper2004}. Our aim is a self-contained derivation of late-time survival bounds by controlling the memory contribution to the killed forward operator. For the lower bound, a positive principal-eigenvector projection must be initialized quantitatively, with enough mass to dominate the truncation errors uniformly in $N$.

The closest discrete-time escape estimate is due to Andr\'e and Zuazn\'abar \cite{AndreZuaznabar2026}, who consider the same ERW started at zero and killed upon reaching $\pm N$. Their Theorem~1 establishes two-sided exponential bounds on the $N^{-2}$ rate scale for $0<p<3/4$, every $N\ge2$, and all discrete times $n\ge1$. The remark immediately following their Theorem~1 extends the upper bound to every $0<p<1$. Our main result is the following.

\begin{theorem}\label{thm:main}
Fix $\rho\in(-1,1)$. There exist constants $N_0\ge3$, $T>0$, $C\ge1$, and $0<c_1\le c_2<\infty$, depending only on $\rho$, such that, for every integer $N\ge N_0$, every $r\in[0,1]$, and every $t\ge TN^3$,
\begin{equation}\label{eq:main}
C^{-1}e^{-c_2t/N^2}
\le\mathbb P(\tau_N>t)
\le Ce^{-c_1t/N^2}.
\end{equation}
\end{theorem}

The constants are independent of $N$, $t$, and the first-step bias $r$. In fact, reflection of the increments shows that the laws of $\nu_N$ and $\tau_N$ are independent of $r$, since the interval is symmetric. The theorem concerns a large-time regime beginning at order $N^3$; it does not assert bounds uniform over all $t\ge0$, or identify a single exact exponential rate. The restriction $p<1$ is essential for this uniform conclusion: for $p=1$, absorption occurs at the $N$-th Poisson jump, and the fixed-$N$ exponential decay rate is $1$.

The discrete-time bounds of \cite{AndreZuaznabar2026} transfer directly to the present clock. Indeed, with $s_n=\mathbb P(\nu_N>n)$ and $s_0=1$, independence gives
\begin{equation}\label{eq:poisson-transfer}
\mathbb P(\tau_N>t)=\mathbb E[s_{n(t)}],\qquad
\mathbb E[e^{-a n(t)/N^2}]
=e^{-t(1-e^{-a/N^2})},\quad a>0.
\end{equation}
The inequalities
\[
\frac{ae^{-a}}{N^2}\le1-e^{-a/N^2}\le\frac a{N^2}
\]
show that the exponential rate scale is preserved uniformly in $N$. The term $n(t)=0$ is accommodated by adjusting the multiplicative constants. Thus the upper bound in Theorem~\ref{thm:main}, and its lower bound for $0<p<3/4$, already follow from their results, with a wider time range. Relative to their stated result, the additional parameter range covered by our lower bound is $3/4\le p<1$, for $t\ge TN^3$.

Qin \cite{QinCover2026} studies bilateral exit times through both mean estimates and tail bounds. Proposition~3.1 gives a mean of order $N^2/\sqrt{\log N}$ at $p=3/4$ and an asymptotic proportional to $N^{5-4p}$ for $3/4<p<1$. Lemma~3.2 and Proposition~3.5 provide upper-tail estimates, and Proposition~3.8 gives normalized superdiffusive exit-tail limits at each fixed rescaled time. Combined with an elementary comparison of the future path law, these estimates also imply a lower bound of the form in Theorem~\ref{thm:main}, even for times beginning at order $N^2\log N$; Remark~\ref{rem:qin-comparison} gives the deduction. This implication uses an additional comparison argument and is not a theorem stated in that form in \cite{QinCover2026}. Accordingly, we emphasize the self-contained proof developed here. These comparisons concern the implications of the available results, rather than a claim of historical priority.\footnote{An earlier version of this work, entitled \emph{A Poincar\'e Inequality and Exponential Decay for the Elephant Random Walk}, was first posted as \arxiv{2606.08884} on 7 June 2026, before the posting of \cite{QinCover2026} on 15 September 2026.} An $N^{-2}$ exponential rate scale in the far tail should not be identified with the scale of the mean exit time.

For each fixed $N$, the forward operator considered below converges as $t\to\infty$ to the simple symmetric random-walk Dirichlet Laplacian. Its Poincar\'e inequality is classical; see \cite{BakryGentilLedoux2014,LevinPeresWilmer2017}. Our proof controls the memory-dependent perturbation and initializes a positive principal-eigenvector projection. The confined-path likelihood estimate in Lemma~\ref{lem:initialization} also yields the lower bound directly, as recorded in Remark~\ref{rem:direct-lower}; the projection argument makes explicit the evolution of that positive mode in the presence of truncation errors. Comparisons of ERW path laws with simple random walk have also been used by Curien and Laulin \cite[Proposition~2]{CurienLaulin2024} in the planar recurrence problem. Here the threshold $N^3$ arises by comparing the bound $N/t$ with the spectral scale $N^{-2}$; we do not establish its optimality. Convergence of distributions conditioned on survival requires additional arguments and is not addressed by Theorem~\ref{thm:main}.

For context, connections with P\'olya urns and martingale methods are developed in \cite{BaurBertoin2016,Bercu2017,ChauvinPouyanneSahnoun2011}; central limit theorems, strong approximation, and related identities appear in \cite{Coletti2017a,Coletti2017b,BercuChabanolRuch2019,Guevara2019}. Scaling limits, reinforced variants, and superdiffusive limit laws are studied in \cite{BercuLaulin2021,Bercu2022stops,GuerinLaulinRaschel2023,GuerinLaulinRaschel2025,GuerinLaulinRaschelPoly,Bertenghi2021,BertenghiRosales2022,Laulin2022martingale}. Results on returns to the origin and on initial training \cite{Bertoin2022zeros,Fang2024,Fang2025} concern stopping events different from bilateral exit.

Multidimensional models, modified memories, delays, and recurrence transitions are considered in \cite{Laulin2022,Qin2025,Qin2025phase,MaulikRoySadhukhan2024,MaulikRoySadhukhan2025,NguyenNgoc2026,GutStadtmuellerVariations,GutStadtmuellerDelays}. Estimation and broader overviews are treated in \cite{BercuLaulin2024,GutStadtmuellerReview,Laulin2022thesis}. Noise-reinforced scaling limits and ERWs on infinite Cayley trees are studied in \cite{Bertoin2020} and \cite{Mukherjee2025}, respectively.

Related random-walk models, learning applications, and mean-field dynamics are discussed in \cite{GuoWu2022,GuoWu2021,Guevara2021,BenaimLeBoudec2008}, while elephant-reinforced neuronal models appear in \cite{Papag26,NajPapSil}. In earlier work \cite{ColettiPapageorgiou2019}, Coletti and Papageorgiou studied recurrence and transience of the ERW and a replica mean-field model under a Poisson hypothesis. That hypothesis concerns asymptotic independence between replicas and is distinct from the independent Poisson clock used here.

Section~2 derives the truncated marginal forward equations and controls their forcing terms and jumps. Section~3 proves the upper bound by a Dirichlet energy estimate, and Section~4 initializes and propagates the principal-eigenvector projection to obtain the lower bound.

\section{Poisson truncation and forward equations}

Fix $\varepsilon\in(0,1)$. We use a half-open truncation window,
\begin{equation}\label{eq:window}
I_t=\{n\ge1:(1-\varepsilon)t<n\le(1+\varepsilon)t\},
\qquad A_t=\{n(t)\in I_t\}.
\end{equation}
This convention makes the deterministic changes of the window right-continuous. It differs from the strict window $|n(t)-t|<\varepsilon t$ only at its upper endpoint. We work at sufficiently large positive times, so $I_t$ is nonempty. Define
\[
P_k(t)=\mathbb P(X_t=k),\qquad
\widehat P_k(t)=\mathbb P(X_t=k,A_t),\qquad k\in D_N.
\]
Since the process is absorbed, an interior position entails survival. Set $P_{\pm N}=\widehat P_{\pm N}=0$ as a Dirichlet convention for the transient vectors, not as an assertion about the actual absorbed masses. Then
\begin{equation}\label{eq:survival-truncation}
\sum_{k\in D_N}\widehat P_k(t)
\le\mathbb P(\tau_N>t)
\le\sum_{k\in D_N}\widehat P_k(t)+\mathbb P(A_t^c).
\end{equation}

\begin{lemma}\label{Ac}
There is $d_\varepsilon>0$ such that, for all $t>0$,
\[
\mathbb P(A_t^c)\le2e^{-d_\varepsilon t}.
\]
\end{lemma}
\begin{proof}
For a Poisson variable $Z$ of mean $t$, $\mathbb E e^{sZ}=\exp(t(e^s-1))$. Chernoff's inequality gives
\[
\mathbb P(Z\ge(1+\varepsilon)t)\le e^{-t h(\varepsilon)},\qquad
\mathbb P(Z\le(1-\varepsilon)t)\le e^{-t h(-\varepsilon)},
\]
where $h(x)=(1+x)\log(1+x)-x$. Both exponents are positive, and $A_t^c$ is contained in the union of these two tail events. Take $d_\varepsilon=\min\{h(\varepsilon),h(-\varepsilon)\}$.
\end{proof}

\begin{lemma}\label{lem:jump-probabilities}
For $k\in D_N$ with $\widehat P_k(t)>0$, let
\[
b_k(t)=\mathbb E\left[\frac1{n(t)}\,\middle|\,X_t=k,A_t\right].
\]
Conditional on $X_t=k$ and $A_t$, the probabilities of the next jump direction are
\begin{equation}\label{eq:conditional-rates}
Q_t^k(\pm1)=\frac12\pm\frac{\rho k}{2}b_k(t),
\qquad
\frac1{(1+\varepsilon)t}\le b_k(t)\le\frac1{(1-\varepsilon)t}.
\end{equation}
\end{lemma}
\begin{proof}
Condition on $n(t)=n\in I_t$ and use \eqref{eq:skeleton}. This gives
\[
b_k(t)=\sum_{n\in I_t}\frac1n
\frac{\mathbb P(n(t)=n,X_t=k)}{\widehat P_k(t)}.
\]
The bounds follow by averaging the bounds on $1/n$ within $I_t$.
\end{proof}
When $\widehat P_k(t)=0$, choose any $b_k(t)$ in the indicated interval; its products with $\widehat P_k(t)$ vanish. For large enough $t$ compared with $N$, all coefficients in \eqref{eq:conditional-rates} are nonnegative, including these arbitrary extensions.

Define the forward matrix on $\mathbb R^{2N-1}$ by
\begin{equation}\label{eq:Lt}
(L_tu)_k=Q_t^{k-1}(1)u_{k-1}+Q_t^{k+1}(-1)u_{k+1}-u_k,
\end{equation}
where terms from outside $D_N$ are omitted. Its coefficients depend on the conditional law, so this is an effective marginal forward operator, not an assertion of the Markov property for $X_t$ alone.

At a window-change time $s$, write
\[
\Delta\widehat P(s)=\widehat P(s)-\widehat P(s-)
\]
for the jump of the right-continuous truncated vector.

\begin{lemma}\label{lem:kolmogorov-system}
Let
\[
\mathcal J=\{t>0:(1-\varepsilon)t\in\mathbb N
\text{ or }(1+\varepsilon)t\in\mathbb N\}.
\]
At sufficiently large times outside $\mathcal J$,
\begin{equation}\label{eq:forward}
\widehat P'(t)=L_t\widehat P(t)+E(t),\qquad
\|E(t)\|_1\le C_\varepsilon e^{-d t},
\end{equation}
for some $d>0$ depending only on $\varepsilon$. At $s\in\mathcal J$,
\begin{equation}\label{eq:jumps}
\|\Delta\widehat P(s)\|_1\le C_\varepsilon e^{-d s}.
\end{equation}
The number of points of $\mathcal J$ in any interval of length one is bounded uniformly.
\end{lemma}
\begin{proof}
Write
\[
w_n(t)=e^{-t}\frac{t^n}{n!},\qquad
q_n(k)=\mathbb P(Y_n=k,\nu_N>n),\qquad k\in D_N.
\]
Independence of the clock yields $\widehat P(t)=\sum_{n\in I_t}w_n(t)q_n$. Let $M_n$ be the killed one-step forward matrix from \eqref{eq:skeleton}, so $q_{n+1}=M_nq_n$ for $n\ge1$.
Outside $\mathcal J$, write $\ell=\min I_t$ and $u=\max I_t$; both are locally constant. Since $w_n'=w_{n-1}-w_n$,
\[
\widehat P'(t)=\sum_{n=\ell}^u w_{n-1}(t)q_n-\widehat P(t).
\]
On the other hand, conditioning on the count in each source state gives
\[
L_t\widehat P(t)=\sum_{n=\ell}^u w_n(t)q_{n+1}-\widehat P(t).
\]
Consequently the error is exactly
\[
E(t)=w_{\ell-1}(t)q_\ell-w_u(t)q_{u+1},
\qquad \|E(t)\|_1\le w_{\ell-1}(t)+w_u(t).
\]
The two indices are $(1\pm\varepsilon)t+O(1)$. For all sufficiently large $t$, they are at distance at least $\varepsilon t/2$ from $t$, so Lemma~\ref{Ac}, with a smaller deviation parameter, bounds these masses by $C_\varepsilon e^{-dt}$.
At an exceptional time, at most two Poisson levels are added or removed; the same estimate bounds the sum of their masses and proves \eqref{eq:jumps}. Finally, $\mathcal J$ is the union of two arithmetic sequences, giving the stated counting bound.
\end{proof}

\section{Dirichlet estimates and the upper bound}

Decompose $L_t=L_0+R(t)$, where
\[
(L_0u)_k=\tfrac12u_{k-1}+\tfrac12u_{k+1}-u_k,
\qquad u_{-N}=u_N=0.
\]
Throughout, $\langle\cdot,\cdot\rangle$ and $\|\cdot\|_2$ denote the Euclidean inner product and norm. The matrix $L_0$ is symmetric on $\mathbb R^{2N-1}$. It is the killed simple-random-walk forward operator; loss of mass at both boundaries makes it nonconservative.

\begin{lemma}\label{lem:spectral}
Put
\[
v_k=\cos\left(\frac{\pi k}{2N}\right),\qquad
\lambda_1=1-\cos\left(\frac{\pi}{2N}\right).
\]
Then $L_0v=-\lambda_1v$, $0<v_k\le1$ for $k\in D_N$, and
\begin{equation}\label{eq:energy}
-\langle f,L_0f\rangle
=\frac12\sum_{k=-N}^{N-1}(f_{k+1}-f_k)^2
\ge\lambda_1\|f\|_2^2
\end{equation}
for every $f$ with $f_{-N}=f_N=0$. Moreover,
\[
\lambda_1=\frac{\pi^2}{8N^2}+O(N^{-4}),
\qquad \frac1{2N^2}\le\lambda_1\le\frac{\pi^2}{8N^2}.
\]
\end{lemma}
\begin{proof}
Expanding the squares and using both boundary values gives the energy identity. The Dirichlet sine basis
\[
f^{(j)}_k=\sin\left(\frac{j\pi(k+N)}{2N}\right),
\qquad 1\le j\le2N-1,
\]
diagonalizes $-L_0$, with eigenvalues $1-\cos(j\pi/(2N))$. The smallest is $\lambda_1$, and its eigenvector is $v$. Expanding $f$ in this orthogonal basis proves the inequality. Taylor's formula gives the asymptotic expansion; the stated bounds follow from $2x^2/\pi^2\le1-\cos x\le x^2/2$ for $0\le x\le\pi/2$.
\end{proof}

\begin{lemma}\label{lem:perturbation-bound}
For sufficiently large $t$, there is $B>0$, depending only on $\rho,\varepsilon$, such that
\begin{align}
|\langle f,R(t)f\rangle|&\le\frac{BN}{t}\|f\|_2^2,\label{eq:quadratic}\\
|\langle R(t)u,v\rangle|&\le\frac{BN}{t}\langle u,v\rangle
\quad\text{for }u\ge0.\label{eq:projection}
\end{align}
\end{lemma}
\begin{proof}
The explicit perturbation is
\[
(R(t)u)_k=\frac\rho2\bigl((k-1)b_{k-1}(t)u_{k-1}
-(k+1)b_{k+1}(t)u_{k+1}\bigr),
\]
with boundary terms omitted. Therefore
\[
|(R(t)u)_k|\le\frac{CN}{t}(|u_{k-1}|+|u_{k+1}|).
\]
Multiplication by $|f_k|$, summation, and $2|ab|\le a^2+b^2$ prove \eqref{eq:quadratic}.
For the second estimate, shifting indices gives
\[
\langle R(t)u,v\rangle
=\frac\rho2\sum_{k\in D_N}k\,b_k(t)u_k
\bigl(v_{k+1}-v_{k-1}\bigr).
\]
For $u\ge0$, use $|v_{k+1}-v_{k-1}|\le v_{k+1}+v_{k-1}$ and the exact identity
\[
v_{k-1}+v_{k+1}=2\cos\left(\frac\pi{2N}\right)v_k\le2v_k,
\]
with $v_{\pm N}=0$. This gives \eqref{eq:projection} after enlarging $B$.
\end{proof}

From now on constants may depend on $\rho,\varepsilon$, but never on $N,t,r$. Let $H(t)=\|\widehat P(t)\|_2^2$. At ordinary times, the preceding lemmas imply
\[
H'(t)\le-2\lambda_1H(t)+\frac{2BN}{t}H(t)+Ce^{-dt},
\]
since $0\le\widehat P_k\le1$ and $|\langle\widehat P,E\rangle|\le\|E\|_1$. Also $H(t)\le1$ and
\[
|\Delta H(s)|\le2\|\Delta\widehat P(s)\|_1\le Ce^{-ds}.
\]
Choose a fixed $T_0$ large enough and put $S=T_0N^3$. For $t\ge S$, the perturbation is absorbed into the negative spectral term, so for some fixed $a>0$,
\[
H'(t)\le-\frac a{N^2}H(t)+Ce^{-dt}.
\]
Increasing a lower threshold $N_0$ if necessary, ensure $a/N^2\le d/2$. Variation of constants, including the jumps, yields
\begin{align*}
H(t)\le{}&H(S)e^{-a(t-S)/N^2}
+C\int_S^t e^{-a(t-s)/N^2}e^{-ds}\,ds\\
&+C\sum_{s\in\mathcal J\cap(S,t]}e^{-a(t-s)/N^2}e^{-ds}.
\end{align*}
Factoring out $e^{-a(t-S)/N^2}$, both remaining error terms are bounded uniformly, because $d-a/N^2\ge d/2$ and $\mathcal J$ has bounded counting density. Thus
\begin{equation}\label{eq:shifted-upper}
H(t)\le C e^{-a(t-S)/N^2},\qquad t\ge S.
\end{equation}
We retain this time shift until imposing $t\ge2S$. For such times, Cauchy--Schwarz and \eqref{eq:survival-truncation} give
\[
\mathbb P(\tau_N>t)
\le2e^{-dt}+C\sqrt N\,e^{-at/(4N^2)}.
\]
Here and below $d$ may be decreased to cover all concentration bounds. Since $t/N^2\ge2T_0N$, the factor $\sqrt N$ is bounded by a constant times $e^{at/(8N^2)}$. Reducing the rate once more if necessary gives
\begin{equation}\label{eq:upper}
\mathbb P(\tau_N>t)\le Ce^{-c_1t/N^2},\qquad t\ge2T_0N^3.
\end{equation}

\section{The lower bound}

With the eigenvector from Lemma~\ref{lem:spectral}, define
\[
\phi(t)=\langle\widehat P(t),v\rangle.
\]
Since $0<v_k\le1$,
\begin{equation}\label{eq:phi-survival}
\phi(t)\le\mathbb P(\tau_N>t,A_t)\le\mathbb P(\tau_N>t).
\end{equation}
We first establish a quantitative initial value at $S=T_0N^3$.

\begin{lemma}\label{lem:initialization}
For fixed $T_0>0$, there is $C_0>0$, depending only on $\rho,\varepsilon,T_0$, such that for all sufficiently large $N$ and every $r\in[0,1]$,
\[
\phi(T_0N^3)\ge e^{-C_0N}.
\]
\end{lemma}
\begin{proof}
Set $S=T_0N^3$ and $M=\lceil N/2\rceil\ge2$. Compare the ERW with a reference walk having the same first-step law and independent symmetric increments thereafter. On any reference path of positive probability, the likelihood ratio for $m$ steps is
\[
D_m=\prod_{n=1}^{m-1}\left(1+\rho\frac{Y_n\sigma_{n+1}}n\right).
\]
We only consider integers $m$ satisfying
\[
(1-\varepsilon/2)S<m<(1+\varepsilon/2)S
\]
and paths in $B_m=\{\max_{0\le j\le m}|Y_j|<M\}$. Such paths satisfy $|Y_n|\le\min(n,N/2)$. Since $|\rho|<1$, the inequality $\log(1+x)\ge x-C_\rho x^2$ for $|x|\le|\rho|$ implies
\begin{equation}\label{eq:likelihood}
\log D_m\ge\rho\sum_{n=1}^{m-1}\frac{Y_n\sigma_{n+1}}n
-C_\rho\sum_{n=1}^{m-1}\frac{Y_n^2}{n^2}.
\end{equation}
Splitting at $n=N$ gives, uniformly on $B_m$,
\[
\sum_{n=1}^{m-1}\frac{Y_n^2}{n^2}
\le N+\frac{N^2}{4}\sum_{n>N}\frac1{n^2}\le CN.
\]
Moreover, $2Y_n\sigma_{n+1}=Y_{n+1}^2-Y_n^2-1$, so summation by parts yields
\begin{align*}
2\sum_{n=1}^{m-1}\frac{Y_n\sigma_{n+1}}n
={}&\frac{Y_m^2}{m-1}-Y_1^2
+\sum_{n=2}^{m-1}\frac{Y_n^2}{n(n-1)}
-\sum_{n=1}^{m-1}\frac1n.
\end{align*}
The boundary term and the nonnegative quadratic sum are bounded by $CN$, using $|Y_n|\le\min(n,N/2)$. Therefore
\[
\left|\sum_{n=1}^{m-1}\frac{Y_n\sigma_{n+1}}n\right|
\le C(N+\log m)\le CN,
\]
where the last inequality uses the stipulated range $m\asymp N^3$ with fixed $T_0$. It follows from \eqref{eq:likelihood} that $D_m\ge e^{-CN}$ on $B_m$.

For completeness, the reference-walk confinement bound follows directly from
\[
h_j=\cos\left(\frac{\pi j}{2M}\right),\qquad
\mu=\cos\left(\frac\pi{2M}\right).
\]
The killed symmetric transition matrix on $\{-M+1,\ldots,M-1\}$ sends $h$ to $\mu h$. The first-step bias has no effect on this calculation from $0$, since $h_1=h_{-1}=\mu$. Thus
\[
\mathbb E_{\rm ref}[h_{Y_m};B_m]=\mu^m,
\qquad
\mathbb P_{\rm ref}(B_m)\ge\mu^m\ge e^{-Cm/N^2}.
\]
The last bound follows from $M\ge N/2$ and $M\ge2$. Consequently
\[
\mathbb P_{\rm ERW}(B_m)\ge e^{-CN}e^{-Cm/N^2}\ge e^{-CN}
\]
throughout the chosen range of $m$. On $B_m$, $v_{Y_m}\ge1/\sqrt2$. Finally, the event
\[
(1-\varepsilon/2)S<n(S)<(1+\varepsilon/2)S
\]
is contained in $A_S$ and has probability at least $1/2$ for sufficiently large $S$, by Poisson concentration. Independence of the clock, followed by averaging over these values of $m$, proves the lemma after increasing the constant.
\end{proof}

At non-exceptional times, symmetry of $L_0$ and Lemma~\ref{lem:perturbation-bound} give
\begin{equation}\label{eq:phi-differential}
\phi'(t)\ge-\left(\lambda_1+\frac{BN}{t}\right)\phi(t)-Ce^{-dt}.
\end{equation}
At every exceptional time,
\[
|\Delta\phi(s)|\le\|\Delta\widehat P(s)\|_1\le Ce^{-ds}.
\]
For $t\ge s\ge S=T_0N^3$, put
\[
K(t,s)=e^{-\lambda_1(t-s)}\left(\frac st\right)^{BN}.
\]
Integrating \eqref{eq:phi-differential} on each continuity interval and including all jumps yields
\begin{equation}\label{eq:phi-integrated}
\phi(t)\ge K(t,S)\phi(S)
-C\int_S^t K(t,s)e^{-ds}\,ds
-C\sum_{s\in\mathcal J\cap(S,t]}K(t,s)e^{-ds}.
\end{equation}
To estimate both error terms, observe that
\[
\frac{K(t,s)}{K(t,S)}
=e^{\lambda_1(s-S)}\left(\frac sS\right)^{BN}
\le\exp\left[\left(\lambda_1+\frac{BN}{S}\right)(s-S)\right].
\]
For sufficiently large $N$, $\lambda_1+BN/S\le d/2$. It follows that
\[
\frac{K(t,s)e^{-ds}}{K(t,S)}
\le e^{-dS}e^{-d(s-S)/2}.
\]
The integral of the right-hand side is bounded by $Ce^{-dS}$; its sum over $\mathcal J$ has the same bound by the counting estimate in Lemma~\ref{lem:kolmogorov-system}. Hence
\[
\phi(t)\ge K(t,S)\bigl(\phi(S)-Ce^{-dS}\bigr).
\]
By Lemma~\ref{lem:initialization}, $\phi(S)\ge e^{-C_0N}$. Increasing $N_0$ if necessary, $Ce^{-dT_0N^3}\le\tfrac12e^{-C_0N}$, and therefore
\begin{equation}\label{eq:lower-refined}
\phi(t)\ge\tfrac12e^{-C_0N}e^{-\lambda_1(t-S)}
\left(\frac St\right)^{BN},\qquad t\ge S.
\end{equation}
For $t\ge S$, use $\log(t/S)\le t/S$ to obtain
\[
BN\log(t/S)\le\frac B{T_0}\frac t{N^2},\qquad
C_0N\le\frac{C_0}{T_0}\frac t{N^2},\qquad
\lambda_1(t-S)\le\frac{\pi^2}{8}\frac t{N^2}.
\]
Combining these inequalities with \eqref{eq:phi-survival} and \eqref{eq:lower-refined} gives
\[
\mathbb P(\tau_N>t)\ge\tfrac12e^{-c_2t/N^2},\qquad t\ge T_0N^3,
\]
for a constant $c_2>0$. Fixing, for example, $\varepsilon=1/2$ and combining this estimate with \eqref{eq:upper} proves Theorem~\ref{thm:main} with $T=2T_0$. All constants can then be chosen to depend only on $\rho$.

\begin{remark}[A direct lower bound from confinement]\label{rem:direct-lower}
The path comparison in Lemma~\ref{lem:initialization} also proves the late-time lower bound without propagating the projection. Indeed, for every integer $m\ge N^3$, retain the harmonic term in the summation-by-parts identity. On the same event $B_m$, this gives
\[
\log D_m\ge-A_\rho N-\frac{|\rho|}{2}\log m.
\]
The reference-walk confinement estimate then yields
\[
\mathbb P(\nu_N>m)
\ge e^{-A_\rho N}m^{-|\rho|/2}e^{-c_*m/N^2},
\qquad m\ge N^3,
\]
with constants independent of $N,m,r$. Since $N\le m/N^2$ and
\[
\frac{N^2\log m}{m}\le\frac{3\log N}{N}\le\frac3e,
\qquad N\ge3,\ m\ge N^3,
\]
the right-hand side is at least $e^{-\kappa_\rho m/N^2}$ after increasing $\kappa_\rho$. For $t\ge2N^3$, restricting the Poisson count to $t/2\le n(t)\le2t$, an event of probability at least $1/2$ for sufficiently large $N$, gives
\[
\mathbb P(\tau_N>t)\ge\tfrac12e^{-2\kappa_\rho t/N^2}.
\]
Thus the confined-path estimate is already sufficient for the lower bound. The preceding spectral proof additionally keeps track of the principal-mode evolution and of both truncation-error contributions.
\end{remark}

\begin{remark}[Comparison with Qin's exit-tail estimates]\label{rem:qin-comparison}
Here we record explicitly an additional deduction from \cite{QinCover2026}, rather than attributing its conclusion to a theorem stated there. Fix $p\in[3/4,1)$, so $\rho=2p-1\in[1/2,1)$, and put
\[
b_p(N)=
\begin{cases}
\sqrt{\log N},&p=3/4,\\
N^{4p-3},&p>3/4.
\end{cases}
\]
For $m=N^2$ and sufficiently large $N$, Qin's estimates imply
\begin{equation}\label{eq:qin-central}
\mathbb P(\nu_N>m,\ |Y_m|\le N/2)\ge\frac{c}{b_p(N)}.
\end{equation}
At criticality, use the event constructed in the proof of \cite[Proposition~3.1(i), equations~(48)--(49)]{QinCover2026}. In the superdiffusive case, apply \cite[Proposition~3.8]{QinCover2026} at radius $a=\lfloor N/2\rfloor$ and fixed rescaled time $5$. Since $5a^2\ge N^2$ for large $N$, survival in $(-a,a)$ to time $5a^2$ implies the event in \eqref{eq:qin-central}. The limiting probability after normalization is positive by the Gaussian representation in \cite[equation~(33)]{QinCover2026}. Reflection extends these statements to every first-step bias $r$.

Condition on $Y_m=x$, $|x|\le N/2$, and compare the future path with a simple symmetric walk $S$ started at $x$ at time $m$. On a reference path that stays in $D_N$ through time $n>m$, the conditional likelihood ratio is
\[
\mathcal D_{m,n}=\prod_{k=m}^{n-1}
\left(1+\rho\frac{S_k(S_{k+1}-S_k)}{k}\right).
\]
Here $|\rho S_k/k|\le1/N$ and
$\sum_{k=m}^{n-1}S_k^2/k^2\le N^2/(m-1)\le2$.
Moreover, summation by parts gives
\begin{align*}
2\sum_{k=m}^{n-1}\frac{S_k(S_{k+1}-S_k)}k
={}&\frac{S_n^2}{n-1}-\frac{x^2}{m}
+\sum_{k=m+1}^{n-1}\frac{S_k^2}{k(k-1)}
-\sum_{k=m}^{n-1}\frac1k.
\end{align*}
Since $\rho\ge0$, discarding the nonnegative quadratic terms in this identity and applying the logarithmic estimate used in Lemma~\ref{lem:initialization} yields
\[
\log\mathcal D_{m,n}\ge-C-\frac\rho2\log(n/m).
\]
For the reference walk, the principal eigenvector gives, uniformly over $|x|\le N/2$,
\[
\mathbb P_{\rm ref}^x\bigl(|S_{m+j}|<N,\ 0\le j\le n-m\bigr)
\ge v_x\cos\!\left(\frac\pi{2N}\right)^{n-m}
\ge c e^{-C(n-m)/N^2}.
\]
Combining this with \eqref{eq:qin-central} proves
\begin{equation}\label{eq:qin-derived-lower}
\mathbb P(\nu_N>n)
\ge\frac{c}{b_p(N)}\left(\frac n{N^2}\right)^{-\rho/2}
e^{-Cn/N^2},\qquad n\ge N^2.
\end{equation}
The case $n=m$ follows directly from \eqref{eq:qin-central}. All constants depend only on $p$. For $n\ge A N^2\log N$, with fixed $A\ge1$, the factors preceding the exponential can be absorbed into a larger exponential rate. Restricting the Poisson count to $t/2\le n(t)\le2t$ transfers the resulting lower bound to $t\ge2A N^2\log N$. In particular, this deduction recovers the late-time lower bound in Theorem~\ref{thm:main}; the proof above establishes that theorem independently of \cite{QinCover2026}.
\end{remark}


\begin{thebibliography}{99}

\bibitem{AndreZuaznabar2026}
M.~Andr\'e and L.~Zuazn\'abar,
\emph{Estimates on Escape Times for the Elephant Random Walk},
\arxiv{2602.18953}, 2026.

\bibitem{QinCover2026}
S.~Qin,
\emph{Cover times and ranges of elephant random walks},
\arxiv{2609.17264}, 2026.

\bibitem{BakryGentilLedoux2014}
D.~Bakry, I.~Gentil, and M.~Ledoux,
\emph{Analysis and Geometry of Markov Diffusion Operators},
Grundlehren der mathematischen Wissenschaften, vol.~348, Springer, 2014.
\doi{10.1007/978-3-319-00227-9}.

\bibitem{BaurBertoin2016}
E.~Baur and J.~Bertoin,
Elephant random walks and their connection to P\'olya-type urns,
\emph{Phys. Rev. E} \textbf{94} (2016), 052134.
\doi{10.1103/PhysRevE.94.052134}.

\bibitem{BenaimLeBoudec2008}
M.~Bena\"im and J.-Y.~Le Boudec,
A class of mean field interaction models for computer and communication systems,
\emph{Performance Evaluation} \textbf{65}(11--12) (2008), 823--838.
\doi{10.1016/j.peva.2008.03.005}.

\bibitem{Bercu2017}
B.~Bercu,
A martingale approach for the elephant random walk,
\emph{J. Phys. A: Math. Theor.} \textbf{51} (2018), 015201.
\doi{10.1088/1751-8121/aa95a6}.

\bibitem{Bercu2022stops}
B.~Bercu,
On the elephant random walk with stops playing hide and seek with the Mittag--Leffler distribution,
\emph{J. Stat. Phys.} \textbf{189} (2022), article~12, 1--27.

\bibitem{BercuChabanolRuch2019}
B.~Bercu, M.-L.~Chabanol, and J.-J.~Ruch,
Hypergeometric identities arising from the elephant random walk,
\emph{J. Math. Anal. Appl.} \textbf{480}(1) (2019), 123360.

\bibitem{BercuLaulin2021}
B.~Bercu and L.~Laulin,
On the center of mass of the elephant random walk,
\emph{Stochastic Process. Appl.} \textbf{133} (2021), 111--128.
\doi{10.1016/j.spa.2020.11.004}.

\bibitem{BercuLaulin2024}
B.~Bercu and L.~Laulin,
How to estimate the memory of the elephant random walk,
\emph{Communications in Statistics---Theory and Methods} \textbf{53}(7) (2024), 2578--2598.
\doi{10.1080/03610926.2022.2139149}.

\bibitem{Bertenghi2021}
M.~Bertenghi,
\emph{Asymptotic normality of superdiffusive step-reinforced random walks},
\arxiv{2101.00906}, 2021.

\bibitem{BertenghiRosales2022}
M.~Bertenghi and A.~Rosales-Ortiz,
Joint invariance principles for random walks with positively and negatively reinforced steps,
\emph{J. Stat. Phys.} \textbf{189} (2022), article~35.
\doi{10.1007/s10955-022-02993-5}.

\bibitem{Bertoin2022zeros}
J.~Bertoin,
Counting the zeros of an elephant random walk,
\emph{Trans. Amer. Math. Soc.} \textbf{375}(8) (2022), 5539--5560.
\arxiv{2105.09569}.

\bibitem{Bertoin2020}
J.~Bertoin,
Universality of noise reinforced Brownian motions,
in \emph{In and Out of Equilibrium 3}, Progress in Probability, vol.~77,
Birkh\"auser, 2021, pp.~147--161.
\doi{10.1007/978-3-030-60754-8_7}.

\bibitem{ChauvinPouyanneSahnoun2011}
B.~Chauvin, N.~Pouyanne, and R.~Sahnoun,
Limit distributions for large P\'olya urns,
\emph{Ann. Appl. Probab.} \textbf{21}(1) (2011), 1--32.
\doi{10.1214/10-AAP696}.

\bibitem{Coletti2017a}
C.~F.~Coletti, R.~Gava, and G.~M.~Sch\"utz,
Central limit theorem and related results for the elephant random walk,
\emph{J. Math. Phys.} \textbf{58}(5) (2017), 053303.
\doi{10.1063/1.4983566}.

\bibitem{Coletti2017b}
C.~F.~Coletti, R.~Gava, and G.~M.~Sch\"utz,
A strong invariance principle for the elephant random walk,
\emph{J. Stat. Mech.: Theory Exp.} (2017), 123207.
\doi{10.1088/1742-5468/aa9680}.

\bibitem{ColettiPapageorgiou2019}
C.~F.~Coletti and I.~Papageorgiou,
Asymptotic analysis of the elephant random walk,
\emph{J. Stat. Mech.: Theory Exp.} (2021), 013205.
\doi{10.1088/1742-5468/abcd36}.

\bibitem{CurienLaulin2024}
N.~Curien and L.~Laulin,
Recurrence of the plane Elephant random walk,
\emph{C. R. Math.} \textbf{362} (2024), 1183--1188.
\doi{10.5802/crmath.620}.

\bibitem{Fang2025}
Z.~Fang,
How Long does it Take to Train an Elephant Random Walk,
\emph{J. Stat. Phys.} \textbf{193} (2026), article~80.
\doi{10.1007/s10955-026-03646-7}.
Earlier version: \emph{How long does it take an elephant random walk to forget its training},
\arxiv{2509.15049}, 2025.

\bibitem{Fang2024}
Z.~Fang,
How often does a critical elephant random walk return to origin,
\emph{Electron. Commun. Probab.} \textbf{29} (2024), article~64, 1--12.
\doi{10.1214/24-ECP636}.

\bibitem{GuerinLaulinRaschelPoly}
H.~Gu\'erin, L.~Laulin, and K.~Raschel,
Elephant polynomials,
\emph{Aequationes Math.} \textbf{99}(2) (2025), 751--766.
\doi{10.1007/s00010-024-01095-9}.

\bibitem{GuerinLaulinRaschel2025}
H.~Gu\'erin, L.~Laulin, K.~Raschel, and T.~Simon,
On the limit law of the superdiffusive elephant random walk,
\emph{Electron. J. Probab.} \textbf{30} (2025), article~102, 1--25.
\doi{10.1214/25-EJP1362}.

\bibitem{GuerinLaulinRaschel2023}
H.~Gu\'erin, L.~Laulin, and K.~Raschel,
A fixed-point equation approach for the superdiffusive elephant random walk,
\emph{Ann. Inst. Henri Poincar\'e Probab. Stat.} \textbf{62}(2) (2026), 973--1005.
\doi{10.1214/24-AIHP1534}.
Preprint: \arxiv{2308.14630}, 2023.

\bibitem{Guevara2021}
V.~H.~V\'azquez Guevara, H.~A.~Cruz Su\'arez, and L.~M.~A.~Ar\'evalo Aguilar,
A strategy to improve learning via a minimal random walk,
\emph{Investigaci\'on Estoc\'astica y Estad\'istica en la Educaci\'on} (2022), 71--83.

\bibitem{Guevara2019}
V.~H.~V\'azquez Guevara,
On the almost sure central limit theorem for the elephant random walk,
\emph{J. Phys. A: Math. Theor.} \textbf{52}(47) (2019), 475201.
\doi{10.1088/1751-8121/ab4b5f}.

\bibitem{GuoWu2022}
J.~Guo and Z.~Wu,
Exploring trapping problem on simplicial networks encoding higher-order interactions,
\emph{Modern Phys. Lett. B} \textbf{36}(11) (2022), 2250019.
\doi{10.1142/S0217984922500191}.

\bibitem{GuoWu2021}
J.~Guo and Z.~Wu,
Uncovering the impact of delay phenomenon on random walks in a family of weighted $m$-triangulation networks,
\emph{Internat. J. Modern Phys. B} \textbf{35}(23) (2021), 2150234.
\doi{10.1142/S0217979221502349}.

\bibitem{GutStadtmuellerReview}
A.~Gut and U.~Stadtm\"uller,
Elephant random walks; a review,
\emph{Ann. Univ. Sci. Budapest., Sect. Comput.} \textbf{54} (2023), 171--198.
\doi{10.71352/ac.54.171}.

\bibitem{GutStadtmuellerVariations}
A.~Gut and U.~Stadtm\"uller,
Variations of the elephant random walk,
\emph{J. Appl. Probab.} \textbf{58}(3) (2021), 805--829.

\bibitem{GutStadtmuellerDelays}
A.~Gut and U.~Stadtm\"uller,
Elephant random walks with delays,
\emph{Rev. Roumaine Math. Pures Appl.} \textbf{67}(1--2) (2022), 51--66.

\bibitem{Laulin2022martingale}
L.~Laulin,
New insights on the reinforced elephant random walk using a martingale approach,
\emph{J. Stat. Phys.} \textbf{186} (2022), article~9.
\doi{10.1007/s10955-021-02834-x}.

\bibitem{Laulin2022}
L.~Laulin,
Introducing smooth amnesia to the memory of the elephant random walk,
\emph{Electron. Commun. Probab.} \textbf{27} (2022), article~54, 1--12.
\doi{10.1214/22-ECP495}.

\bibitem{Laulin2022thesis}
L.~Laulin,
\emph{About the elephant random walk}, PhD thesis, Universit\'e de Bordeaux, 2022.
\href{https://theses.hal.science/tel-03708512}{HAL: tel-03708512}.

\bibitem{LevinPeresWilmer2017}
D.~A.~Levin and Y.~Peres, with contributions by E.~L.~Wilmer,
\emph{Markov Chains and Mixing Times}, second edition,
American Mathematical Society, 2017.
\doi{10.1090/mbk/107}.

\bibitem{MaulikRoySadhukhan2025}
K.~Maulik, P.~Roy, and T.~Sadhukhan,
\emph{Phase transitions for elephant random walks with two memory channels},
\arxiv{2509.10225}, 2025.

\bibitem{MaulikRoySadhukhan2024}
K.~Maulik, P.~Roy, and T.~Sadhukhan,
\emph{Asymptotic properties of generalized elephant random walks},
\arxiv{2406.19383}, 2024.

\bibitem{Mukherjee2025}
S.~S.~Mukherjee,
\emph{Elephant random walks on infinite Cayley trees},
\arxiv{2509.03048}, 2025.

\bibitem{NajPapSil}
F.~A.~Najman, I.~Papageorgiou, and S.~K.~C.~A.~da Silveira,
\emph{Spiking Neural Networks with Elephant Reinforcement},
\arxiv{2608.12839}, 2026.

\bibitem{NguyenNgoc2026}
Ngo Phuoc Nguyen Ngoc,
\emph{Recurrence, Transience, and the Rate of Escape for Elephant Random Walks with Two Memory Channels},
\arxiv{2609.06641}, 2026.

\bibitem{Papag26}
I.~Papageorgiou,
\emph{Elephant Reinforced Galves-L\"ocherbach Networks},
\arxiv{2608.10183}, 2026.

\bibitem{SchutzTrimper2004}
G.~M.~Sch\"utz and S.~Trimper,
Elephants can always remember: Exact long-range memory effects in a non-Markovian random walk,
\emph{Phys. Rev. E} \textbf{70} (2004), 045101.
\doi{10.1103/PhysRevE.70.045101}.

\bibitem{Qin2025phase}
S.~Qin,
Recurrence-Transience phase transition of the step-reinforced random walk at $1/2$,
\emph{Probab. Theory Related Fields} \textbf{194} (2026), 485--540.
\doi{10.1007/s00440-025-01432-2}.
Published online in 2025; preprint circulated as
\emph{Step-reinforced random walks and one-half}, \arxiv{2402.16396}, 2024.

\bibitem{Qin2025}
S.~Qin,
Recurrence and transience of multidimensional elephant random walks,
\emph{Ann. Probab.} \textbf{53}(3) (2025), 1049--1078.
\doi{10.1214/24-AOP1728}.

\end{thebibliography}
\end{document}